\documentclass{article}

\usepackage{amssymb}
\usepackage{latexsym}
\usepackage{extarrows}
\usepackage[dvips]{graphicx}

\newtheorem{theo}{Theorem}[section]

\newtheorem{lemma}[theo]{Lemma}

\newtheorem{coro}[theo]{Corollary}

\newtheorem{prop}[theo]{Proposition}
\newtheorem{fact}[theo]{Fact}
\newtheorem{defi}[theo]{Definition}
\newtheorem{algorithm}[theo]{Algorithem}

\newcommand{\qed}{\hspace*{\fill} \rule{7pt}{7pt}}

\def\cF{\mathcal F}

\def\cK{\mathcal K}

\def\endproofbox{\hskip 1.3em\hfill\rule{6pt}{6pt}}

\newenvironment{proof}%
{%
\noindent{\it Proof.}
}%
{%
 \quad\hfill\endproofbox\vspace*{2ex}
}

\def\qed{\hskip 1.3em\hfill\rule{6pt}{6pt}}

\begin{document}
\date{September 28, 2016}
\title{Tur\'an numbers of extensions of some sparse hypergraphs via Lagrangians}
\author{ {Tao Jiang \thanks{ Department of Mathematics, Miami University, Oxford, OH, USA. Email: jiangt@miamioh.edu. \ Research supported in part by National Science Foundation grant DMS-1400249.  The research was done during the author's visit of Hunan University, whose hospitality is gratefully acknowledged.}} \and Yuejian Peng \thanks{ Institute of Mathematics, Hunan University, Changsha, 410082, P.R. China. Email: ypeng1@hnu.edu.cn \ Supported in part by National Natural Science Foundation of China (No. 11271116).} \and Biao Wu \thanks{ College of Mathematics and Econometrics, Hunan University, Changsha 410082, P.R. China. Email: wubiao@hnu.edu.cn.}
}

\maketitle

\begin{abstract}
Given a positive integer $n$ and an $r$-uniform hypergraph (or {\it $r$-graph} for short) $F$, 
the {\it Tur\'an number} $ex(n,F)$ of $F$ is the maximum number of edges in an $r$-graph on $n$ vertices 
that does not contain $F$ as a subgraph. The {\it extension} $H^F $ of $F$ is obtained as follows:
For each pair of vertices $v_i,v_j$ in $F$ not contained in an edge of $F$, we add a set $B_{ij}$ of $r-2$ new vertices 
and the edge $\{v_i,v_j\} \cup B_{ij}$, where the $B_{ij}$ 's are pairwise disjoint over all such pairs $\{i,j\}$.
Let $K^r_p$ denote the complete $r$-graph on $p$ vertices.
For all sufficiently large $n$, we determine the Tur\'an numbers of the extensions of a $3$-uniform $t$-matching, 
a $3$-uniform linear star of size $t$,  and a $4$-uniform linear star of size $t$, respectively. We also show that 
the unique extremal hypergraphs are balanced blowups of $K^3_{3t-1}, K^3_{2t}$, and $K^4_{3t}$, respectively.
Our results generalize the recent result of Hefetz and Keevash \cite{HK}.
\end{abstract}


\section{ Notations and definitions}

For a set $V$ and a positive integer $r$ we denote by $V^{(r)}$ the family of all $r$-subsets of $V$. An {\em $r$-uniform graph} or {\em $r$-graph $G$} consists of a set $V(G)$ of vertices and a set $E(G) \subseteq V(G) ^{(r)}$ of edges. When there is no confusion, we simply write $G$ for $E(G)$.
Let $|G|$ denote the number of edges of $G$. An edge $e=\{a_1, a_2, \ldots, a_r\}$ will be simply denoted by $a_1a_2 \ldots a_r$. An $r$-graph $H$ is  a {\it subgraph} of an $r$-graph $G$, denoted by $H\subseteq G$, if $V(H)\subseteq V(G)$ and $E(H)\subseteq E(G)$.
A subgraph of $G$ {\em induced} by $V'\subseteq V$, denoted as $G[V']$, is the $r$-graph with vertex set $V'$ and edge set $E'=\{e\in E(G):e \subseteq V'\}$.
Let $K^{r}_t$ denote the {\em complete $r$-graph} on $t$ vertices, that is, the $r$-graph on $t$ vertices containing all $r$-subsets of the vertex set as edges.
Let $T_{m}^{r}(n)$ be the balanced blow-up of $K_{m}^r$ on $n$ vertices, i.e., $V(T_{m}^{r}(n))=V_1 \cup V_2 \cup \dots \cup V_m$ such that $V_i \cap V_j=\emptyset$ for every $1\le i < j \le m$ and $|V_1| \le |V_2| \le \dots \le |V_m|\le |V_1|+1$, and $E(T_{m}^{r}(n))=
\{S\in \binom{[n]}{r}: \forall i\in [m], |S\cap V_i|\leq 1\}$. The graph $T_m^r(n)$ is also
commonly called \emph{the $r$-uniform $m$-partite Tur\'an graph on $n$ vertices}.
Let $t_{m}^{r}(n)=|T_{m}^{r}(n)|$. For a positive integer $n$, we let $[n]$ denote $\{1, 2, 3, \ldots, n\}$.
Given positive integers $m$ and $r$, let $[m]_r = m(m-1)\dots(m-r+1)$.

Given an $r$-graph $F$,  an $r$-graph $G$ is called \emph{$F$-free} if it does not contain $F$ as a subgraph. For a fixed positive integer $n$ and an $r$-graph $F$, the {\em Tur\'an number} of $F$, denoted by $ex(n,F)$, is the maximum number of edges in an $r$-graph on $n$ vertices that does not contain $F$ as a subgraph.
 An averaging argument of Katona, Nemetz and Simonovits \cite{KNS} shows that the sequence ${ ex(n,F) \over {n \choose r } }$ is a non-increasing sequence of real numbers in $[0,1]$. Hence, $\lim_{n\rightarrow\infty} { ex(n,F) \over {n \choose r } }$ exists. The {\em Tur\'{a}n density} of $F$ is defined as $$\pi(F)=\lim_{n\rightarrow\infty} { ex(n,F) \over {n \choose r } }.$$

In this paper, we extend the work of Hefetz and Keevash in \cite{HK} and determine Tur\'an numbers of several classes of $r$-graphs using so-called hypergraph Lagrangian method.

\begin{defi}
Let $G$ be an $r$-graph on $[n]$ and let
  $\vec{x}=(x_1,\ldots,x_n) \in [0,\infty)^n$. For every subgraph $H\subseteq G$,
define
$$\lambda (H,\vec{x})=\sum_{e \in E(H)}\prod\limits_{i\in e}x_{i}.$$
\end{defi}

The {\em Lagrangian} of
$G$, denoted by $\lambda (G)$, is defined as
 $$\lambda (G) = \max \{\lambda (G, \vec{x}): \vec{x} \in \Delta \},$$
where $$\Delta=\{\vec{x}=(x_1,x_2,\ldots ,x_n) \in [0,\infty)^{n}: \sum_{i=1}^{n} x_i =1 \}.$$

The value $x_i$ is called the {\em weight} of the vertex $i$ and a vector $\vec{x} \in {\Delta}$ is called a {\em feasible weight vector} on $G$.
A feasible vector  $\vec{y}\in {\Delta}$ is called an {\em optimum weight vector} on $G$ if $\lambda (G, \vec{y})=\lambda(G)$.


Given an $r$-graph $F$, we define the {\em Lagrangian density } $\pi_{\lambda}(F)$ of $F$ to be
$$\pi_{\lambda}(F)=\sup \{r! \lambda(G): F \nsubseteq G\}.$$

\begin{prop}
$\pi(F)\le \pi_{\lambda}(F).$ \qed
\end{prop}
\begin{proof}
Let $\varepsilon>0$ be arbitrary. Let $n$ be large enough and let $G_n$ be a maximum $F$-free $r$-graph on $n$ vertices. We have
$$\pi(F)\leq { |G_n| \over {n \choose r } }+\varepsilon/2
\le r! \sum_{e \in E(G_n)}{1 \over n^r}+\varepsilon =r!\lambda(G_n, ({1\over n}, {1\over n}, \ldots, {1\over n})) +\varepsilon
\le r! \lambda(G_n)+\varepsilon \leq \pi_{\lambda}(F)+\varepsilon.$$
\end{proof}

The Lagrangian method for hypergraph Tur\'an problems were developed independently by
Sidorenko \cite{Sidorenko-87} and Frankl-F\"uredi  \cite{FF}, generalizing work of Motzkin and Straus \cite{MS} and
Zykov \cite{Z}. More recent developments of the method were obtained by Pikhurko \cite{Pikhurko} and Norin and Yepremyan \cite{NY}.
Based on these developments, Brandt, Irwin, and Jiang \cite{BIJ}, and independently Norin and Yepremyan \cite{NY2}
were able to determine the Tur\'an numbers of a large family of hypergraphs and thereby extending earlier works in \cite{bollobas,
FF-F5, mubayi, MP, p, Sidorenko-89}. The methods used by the two groups are quite different. The former group used Pikhurko's
stability method while the latter group used a refined stability method that they developed in \cite{NY}.
In this paper, we extend a recent work on the topic by Hefetz and Keevash \cite{HK} on Lagrangians of intersecting $3$-graphs
to determine the maximum Lagrangian of a $3$-graph not containing a matching of a given size. We also determine
the maximum Lagrangian of a $3$-graph not containing a linear star of a given size and the maximum Lagrangian of
a $4$-graph not containing a linear star of a given size. These results combined with the corresponding general theorems
in \cite{BIJ} and \cite{NY2} then allow us to determine the Tur\'an numbers of some corresponding hypergraphs, which we now define as below.

We say that a pair of vertices $\{i, j\}$ is {\em covered}  in a hypergprah $H$ if there exists $e\in H$ such that  $\{i, j\}\subseteq e$.
 Let $r\ge 3$ and $F$ be an $r$-graph. Let $p\ge |V(F)|$. Let $\cK_p^F $ denote the family of $r$-graphs $H$ that contains a set $C$ of $p$ vertices, called the {\em core}, such that the subgraph of $H$ induced by $C$ contains a copy of $F$ and such that every pair in $C$ that are not
 covered by $F$ is covered by an edge of $H$. We call $\cK_p^F$ the family of {\em
 weak extensions} of $F$ for the given $p$. If $p=|V(F)|$, then
we simply call $\cK_p^F$ the family of {\em extensions} of $F$.
Let $H_p^F $ be a member of ${\mathcal{K}}_p^F $obtained as follows. Label the vertices of $F$ as $v_1,\dots,v_{|V(F)|}$. Add new vertices $v_{|V(F)|+1},\dots,v_{p}$. Let $C=\{v_{1},\dots,v_{p}\}$. For each pair of vertices $v_i,v_j \in C$ not covered in $F$, we add a set $B_{ij}$ of $r-2$ new vertices and the edge $\{v_i,v_j\} \cup B_{ij}$, where the $B_{ij}$ 's are pairwise disjoint over all such pairs $\{i,j\}$. We call $H_p^F $ the {\em extension} of $F$ for the given $p$. If $p=|V(F)|$, then
we simply call $H_p^F$ the {\em extension} of $F$.

Let $r,t$ be integers such that $r\ge 3$ and $t\ge 2$. The {\em $r$-uniform $t$-matching}, denoted by $M^r_t$, is the $r$-graph with $t$ pairwise disjoint edges.
The {\em $r$-uniform linear star of size $t$},
denoted by $L^r_t$, is  the $r$-graph with $t$ edges  such that these $t$ edges contain a common vertex  $x$ but
are pairwise disjoint outside $\{x\}$.

In \cite{HK}, Hefetz and Keevash determined  the Lagrangian density of $M_2^3$ and the Tur\'an number of the extension of $M_2^3$ for all sufficiently large $n$.
In this paper, we generalize their result to determine the Lagrangian density of $M^3_t$ for all $t\ge 2$. We also determine the Lagrangian densities of $L_t^3$ or $L_t^4$,
for all $t\geq 2$.
For each of the hypergraphs mentioned above, we determine the Tur\'an numbers of their extensions for all sufficiently large $n$.
Our method differs from the one employed by Hefetz and Keevash \cite{HK}.  For the matching problem, we use compression and induction.
This allows us to obtain a short proof of the main result of \cite{HK} and solve the problem for general $t$.
We solve the linear star problem for $r=3,4$ by first studying a local version of the matching problem for $r=2,3$, respectively.

\section{Preliminaries}

In this section, we develop some useful properties of Lagrangian functions.
The following fact follows immediately from the definition of the Lagrangian.
\begin{fact}\label{mono}
Let $G_1$, $G_2$ be $r$-graphs and $G_1\subseteq G_2$. Then $\lambda (G_1) \le \lambda (G_2).$
\end{fact}

Given an $r$-graph $G$ and a set $S$ of vertices, the {\em link graph} of $S$ in $G$, denoted by $L_G(S)$, is the hypergraph with edge set $\{e\in {V(G)\setminus S \choose r-|S|}: e \cup S \in E(G)\}$. When $S$ has only one element, e.g. $S=\{i\}$, we write $L_G(i)$ for $L_G(\{i\})$. Furthermore, when there is no confusion,  we will drop the subscript $G$. Given $i,j\in V(G)$, define
$$L_G(j \setminus i)=\{f\in \binom{V(G)\setminus \{i,j\}}{r-1}: f \cup \{j\} \in E(G) {\rm \ and \ } f \cup \{i\} \notin E(G)\},$$ and define
$$\pi_{ij}(G)=\left(E(G)\setminus \{f\cup \{j\}: f\in  L_G(j \setminus i)\} \right) \bigcup \{ f\cup \{i\}: f \in L_G(j \setminus i) \}.$$
By the definition of $\pi_{ij}(G)$, it's straightforward to verify the following fact.

\begin{fact}\label{compression-preserve}
Let $G$ be an $r$-graph on the vertex set $[n]$. Let $\vec{x}=(x_1,x_2,\dots,x_n)$ be a feasible weight vector on $G$. If $x_i \ge x_j$, then $\lambda(\pi_{ij}(G),\vec{x})\ge \lambda(G,\vec{x})$.
\end{fact}
Part (a) of the following lemma is well-known (see \cite{frankl-survey} for instance). We include a short proof of it for completeness.

\begin{lemma}\label{t-free}
Let $r,t\geq 2$ be integers.
Let $G$ be a $M_t^r$-free $r$-graph on the vertex set $[n]$. Let $i,j$ be a pair of vertices, then the following hold:
\newline
{\rm (a)} $\pi_{ij}(G)$ is $M_t^r$-free.
\newline
{\rm (b)} If $G$ is also $K^r_{tr-1}$-free and $\{i,j\}$ is contained in an edge of $G$, then $\pi_{ij}(G)$ is $K^r_{tr-1}$-free.
\end{lemma}
\begin{proof}
Suppose for contradiction that there exist $i,j$ such that $\pi_{ij}(G)$ contains a $t$-matching $M$. Then there must be an edge $e$ of $M$ that
is in $\pi_{ij}(G)$ but not in $G$. This implies that $i\in e$, $j\notin e$ and $e'=(e\setminus\{i\}) \cup\{j\}\in  G$.
If $j$ is not covered by any edge of $M$, then $(M\setminus \{e\})\cup \{e'\}$ is a $t$-matching in $G$,
contradicting $G$ being $M_t^r$-free. Hence, $\exists f\in M$
such that $j\in f$. Let $f'=(f\setminus \{j\})\cup \{i\}$.
By the definition of $\pi_{ij}(G)$, $f$ and $f'$ must both exist in $G$, or else $f$ wouldn't be in $\pi_{ij}(G)$.
But now, $(M\setminus \{e,f\})\cup \{e',f'\}$ is a $t$-matching in $G$, contradicting
$G$ being $M^r_t$-free.

Next, suppose that $G$ is  $K^r_{tr-1}$-free and $\{i,j\}$ is contained in some edge $e$ of $G$. Suppose for contradiction that $\pi_{ij}(G)$ contains a copy $K$
of $K^r_{tr-1}$. Clearly $V(K)$ must contain $i$. If $V(K)$ also contains $j$ then it is easy to see that $K$ also exists in $G$, contradicting $G$ being
$K^r_{tr-1}$-free.  All the edges in $K$ not containing $i$ also exist in $G$.
By our assumption, $V(K)$ contains at least $tr-1-(r-1)=(t-1)r$ vertices outside $e$.
So $K$ contains a $(t-1)$-matching $M$ disjoint from $e$, all of which lie in $G$ by earliest discussion. Now, $M\cup \{e\}$
is a $t$-matching in $G$, a contradiction.
\end{proof}

Next, we show that for $r=2$, part (b) of Lemma \ref{t-free} holds even without the assumption that $\{i,j\}$ is contained in an edge.

\begin{lemma} \label{t-free-2-uniform}
Let $t\geq 2$. Let $G$ be an $M^2_t$-free and $K^2_{2t-1}$-free graph on $[n]$ and $i,j\in [n]$. Then $\pi_{ij}(G)$ is also $K^2_{2t-1}$-free.
\end{lemma}
\begin{proof}
Suppose for contradiction that $\pi_{ij}(G)$ contains a copy $K$ of $K^2_{2t-1}$.  Then  $\pi_{ij}(G)\neq G$ and $K$ contains $i$
but not $j$ (note that $\pi_{ij}$ does not change the common link of $i$ and $j$). Since $\pi_{ij}(G)\neq G$,
$L_G(j\setminus i)\neq \emptyset$. Also, $L_G(i\setminus j)\neq \emptyset$, since otherwise $K\subseteq G$.
Let $a\in V(L_G(i\setminus j))$, $b\in V(L_G(j\setminus i))$. Note that any edge in $\pi_{ij}(G)$ not containing $i$
also exist in $G$. Hence, $K-\{i,a,b\}$ is a complete graph on at least $2t-4$ vertices in $G$, which contains a
$(t-2)$-matching $M$. Now, $M\cup \{ia,jb\}$ is a $t$-matching in $G$, a contradiction.
\end{proof}

\medskip

An $r$-graph $G$ is {\em dense} if for every subgraph $G'$ of $G$ with $|V(G')|<|V(G)|$ we have $\lambda (G') < \lambda (G)$. This is equivalent to saying that all optimum weight vectors on $G$ are in the interior of ${\Delta}$, which means that no coordinate in an optimum weight vector is zero. We say that a hypergraph $G$ \emph{covers pairs}
if every pair of its vertices is covered by an edge.

\begin{fact} {\em (\cite{FR})}\label{dense}
Let $G=(V,E)$ be a dense $r$-graph. Then $G$ covers pairs.
\end{fact}

\begin{defi}
{\rm  Let  $G$ be an $r$-graph on $[n]$ and a linear order $\mu$ on $[n]$. We say that $G$ is \emph{left-compressed}
(or simply \emph{compressed}) relative to $\mu$ if for all $i,j\in [n]$ with $i<_\mu j$ we have $\pi_{ij}(G)=G$.
Let  $\vec{x}$ be a feasible weight vector on $G$. We say that $G$ is \emph{$\vec{x}$-compressed}
if there exists a linear order $\mu$ on $V(G)$  such that $\forall i,j\in V(G)$ whenever $i<_\mu j$ we have  $x_i\geq x_j$ and that $G$ is
left-compressed relative to $\mu$.}
\end{defi}

\begin{algorithm}\label{left-compression}
{\em Let $G$ be an $r$-graph on $[n]$. Let $\vec{x}$ be an optimum weight vector of $G$. If there exist vertices  $i,j$, where $i<j$, such that $x_i> x_j$ and $L_G(j \setminus i)\neq \emptyset $, then replace $G$ by $\pi_{ij}(G)$, continue this process until no such pair exists.
}
\end{algorithm}
In the above algorithm, by relabelling the vertices if necessary, we may assume that $x_1\geq x_2\dots\geq x_n$.
Note that  $s(G)=\sum_{e\in G} \sum_{i\in e} i$ is a positive integer that decreases by at least $1$ in each step.
Hence the algorithm terminates after finite many steps.

\begin{algorithm}\label{left-compression1}{\rm (Dense and compressed subgraph)}

\noindent{\bf Input:}   An $r$-graph $G$.

\noindent{\bf Output:}  A dense subgraph $G'\subseteq G$ together with an optimum weight vector $\vec{y}$
such that $\lambda(G',\vec{y})=\lambda(G)$ and that $G'$ is $\vec{y}$-compressed.

\noindent{\bf Step 1.} If $G$ is not dense, then replace $G$ by a dense subgraph with the same Lagrangian. Otherwise, go to Step 2.

\noindent{\bf Step 2.} Let  $\vec{y}$ be an optimum weight vector of $G$. If $G$ is $\vec{y}$-compressed, then terminate. Otherwise, there exist vertices  $i,j$, where $i<j$, such that $y_i> y_j$ and $L_G(j \setminus i)\neq \emptyset $, then replace $G$ by $\pi_{ij}(G)$ and go to step 1.
\end{algorithm}
Note that the algorithm terminates after finite many steps since  Step 1  reduces the number of vertices by at least 1 in each step and Step 2 reduces the
parameter $s(G)$ (similarly defined as above) by at least $1$ in each step.

\begin{lemma} \label{compression-lemma}
Let $G$ be a $M^r_t$-free $r$-graph and $\vec{x}$ a feasible weight vector on $G$. Then the following hold:
\newline
{\rm (a)} There exists a $M^r_t$-free $r$-graph $H$ with $V(H)=V(G)$
such that $\lambda(H,\vec{x})\geq \lambda(G,\vec{x})$ and that $H$ is $\vec{x}$-compressed.
\newline
{\rm (b)} There exists a dense $M^r_t$-free $r$-graph $G'$ with $\vert V(G')\vert \le \vert V(G)\vert$ together with
an optimum weight vector $\vec{y}$ such that $\lambda(G',\vec{y})=\lambda(G')\geq \lambda(G)$ and that $G'$ is $\vec{y}$-compressed.
Furthermore, if $G$ is $K^r_{tr-1}$-free,
then $G'$ is $K^r_{tr-1}$-free.
\end{lemma}
\begin{proof}
For (a), we apply  Algorithm \ref{left-compression}  to $G$ and let $H$ be the final graph
obtained. That $\lambda(H,\vec{x})\geq \lambda(G,\vec{x})$ follows from Fact \ref{compression-preserve}.
That $H$ is $M^r_t$-free follows from Lemma \ref{t-free}.
That $H$ is $\vec{x}$-compressed follows from the fact that algorithm terminates after finite steps
and it only terminates when the $r$-graph becomes compressed.

For  (b), we apply Algorithm \ref{left-compression1} to $G$ and let $G'$ be the final graph and $\vec{y}$
the optimum weight vector on $G$ implied by the algorithm.  Since Algorithm terminates after finite
many steps, $G'$ and $\vec{y}$ are well-defined. By Fact \ref{compression-preserve},
$\lambda(G')\geq \lambda(G)$. By Lemma \ref{t-free}, $G'$ is $M^r_t$-free. By the algorithm,
$G'$ is $\vec{y}$-compressed. Assume that $G$ is $K^r_{tr-1}$-free. In the process of obtaining
$G'$ we always take a dense subgraph first before applying a compression $\pi_{ij}$. Taking a subgraph
preserves $K^r_{tr-1}$-free condition. For a dense graph, by Lemma \ref{t-free} part (b) performing
$\pi_{ij}$ preserves $K^r_{tr-1}$-free condition. So $G'$ is $K^r_{tr-1}$-free.
\end{proof}

In \cite{MS}, Motzkin and Straus determined the Lagrangian of any given $2$-graph.

\begin{theo} {\em (Motzkin and Straus \cite{MS})} \label{MStheo}
If $G$ is a $2$-graph in which a maximum complete subgraph has  $t$ vertices, then
$\lambda(G)=\lambda(K_t^2)={1 \over 2}(1 - {1 \over t})$.\qed
\end{theo}

Let $G$ be an $r$-graph on $[n]$ and $\vec{x}=(x_1,x_2,\dots,x_n)$ be a weight vector on $G$.
If we view $\lambda(G,\vec{x})$ as a function in variables $x_1,\dots, x_n$, then
$$ \frac{\partial \lambda (G, \vec{x})}{\partial x_i}=\sum_{i \in e \in E(G)}\prod\limits_{j\in e\setminus \{i\}}x_{j}.$$
We sometimes write  $ \frac{\partial \lambda}{\partial x_i}$ for $ \frac{\partial \lambda (G, \vec{x})}{\partial x_i}$.
\begin{fact} {\em (\cite{FR})}\label{fact2}
Let $G$ be an $r$-graph on $[n]$. Let $\vec{x}=(x_1,x_2,\dots,x_n)$ be an optimum weight vector on  $G$. Then
$$ \frac{\partial \lambda (G, \vec{x})}{\partial x_i}=r\lambda(G)$$
for every $i \in [n]$ with $x_i>0$.
\end{fact}

\begin{fact}\label{symmetry}
Let $G$ be an $r$-graph on $[n]$. Let $\vec{x}=(x_1,x_2,\dots,x_n)$ be a feasible weight vector on $G$. Let $i,j\in [n]$, where $i\neq j$.
Suppose that $L_G(i \setminus j)=L_G(j \setminus i)=\emptyset$. Let
$\vec{y}=(y_1,y_2,\dots,y_n)$ be defined by letting $y_\ell=x_\ell$ for every $\ell \in [n]\setminus \{i,j\}$ and letting $y_i=y_j={1 \over 2}(x_i+x_j)$.
Then $\lambda(G,\vec{y})\geq \lambda(G,\vec{x})$. Furthermore, if the pair $\{i,j\}$ is not covered by any edge of $G$ and $\lambda(G,\vec{y})=\lambda(G,\vec{x})$, then $x_i=x_j$.
\end{fact}
\begin{proof}
Since $L_G(i \setminus j)=L_G(j \setminus i)=\emptyset$, we have
$$\lambda(G,\vec{y})-\lambda(G,\vec{x})=\sum_{\{i,j\} \subseteq e \in G}\left[{(x_i+x_j)^2 \over 4}-x_ix_j\right]\prod\limits_{k\in e\setminus \{i,j\}}x_k \ge 0.$$
If the pair $\{i,j\}$ is not covered by any edge of $G$ then equality holds only if $x_i=x_j$.
\end{proof}

As usual, if $V_1,\ldots, V_s$ are disjoint sets of vertices then $\Pi_{i=1}^s V_i=
V_1\times V_2\times\ldots \times V_s=\{(x_1,x_2,\ldots, x_s): \forall i=1,\ldots,s, x_i\in V_i\}$.
We will use $\Pi_{i=1}^s V_i$ to also denote the set of the corresponding unordered $s$-sets.
If $L$ is a hypergraph on $[m]$, then a {\it blowup} of $L$ is a hypergraph $G$
whose vertex set can be partitioned into $V_1,\ldots, V_m$ such that
$E(G)=\bigcup_{e\in L} \prod_{i\in e} V_i$.
The following proposition follows immediately from the definition and is implicit in many papers (see \cite{Keevash} for instance).
\begin{prop} \label{lag-bound}
Let $r\geq 2$. Let $L$ be an $r$-graph and $G$ a blowup of $L$. Suppose $|V(G)|=n$. Then $|G|\leq \lambda(L) n^r$. \qed
\end{prop}


\section{Lagrangian of an $r$-graph not containing a $t$-matching and related Tur\'an numbers}

\subsection{Lagrangian density of $M_t^3$}
\begin{lemma}\label{compression-dense}
Let $n,r,t$ be positive integers where $t\geq 2$ and $n\geq r\geq 2$. Let $\cF$ denote the family of
all $r$-graphs $H$ with no isolated vertex on at most $n$ vertices such that $H$ is $M^r_t$-free and $H\neq K^r_{tr-1}$.
Then there exists a dense $r$-graph $G\in \cF$ and an optimum vector $\vec{x}$ on $G$ such
that $\lambda(G,\vec{x})=\max\{\lambda(H): H\in \cF\}$ and that $G$ is $\vec{x}$-compressed.
\end{lemma}
\begin{proof}
First note that if $H\in \cF$ then $H$ is $K^r_{tr-1}$-free. Otherwise suppose $H$ contains a copy $K$ of $K^r_{tr-1}$.
Then since $H$ has no isolated vertex and $H\neq K^r_{tr-1}$, $H$ contains some edge not in $K$, in which case
we can find a $t$-matching in $H$, a contradiction.
Let $\lambda^*=\max\{\lambda(H): H\in \cF\}$. Let $G_1 \in \cF$ be an $r$-graph with $\lambda(G_1)=\lambda^*$.
By Lemma \ref{compression-lemma} {\rm (b)}, there exists a $M^r_t$-free dense $r$-graph $G'_1$ with $\vert V(G'_1)\vert \le \vert V(G_1)\vert $ such that $\lambda(G'_1)\geq \lambda(G_1)$ and $G'_1$ is $\vec{x}$-compressed, where $\vec{x}$ is an optimum vector of $G'_1$. Furthermore, $G'_1$ is $K^r_{tr-1}$-free. Hence $G'_1\in \cF$. So $\lambda(G'_1)=\lambda^*$. The claim thus holds by letting $G=G'_1$.
\end{proof}

\medskip

Hefetz and Keevash \cite{HK} established the Lagrangian density of $M_2^3$. We give a short new proof here.

\begin{theo}{\em (\cite{HK})}\label{theoHK}
Let $G$ be an $M_{2}^3$-free $3$-graph. Then $\lambda(G)\leq \lambda(K_5^3)=\frac{2}{25}$. Furthermore, if $G\neq K^3_5$ and
$G$ has no isolated vertex,
then $\lambda(G) \le \lambda(K^3_5)-10^{-3}$.
\end{theo}
\begin{proof} (new proof)
It suffices to prove that if $G$ is an $M_2^3$-free $3$-graph with no isolated vertex and $G\neq K^3_5$ then
$\lambda(G) \le \lambda(K^3_5)-10^{-3}$. By Lemma \ref{compression-dense}, it suffices to assume that $G$
is dense and has an optimum weight vector $\vec{x}$ such that $G$ is $\vec{x}$-compressed.
Suppose $V(G)=[n]$.  If $n\leq 5$, then $\lambda(G)\leq  \lambda(K^{3-}_5) <\lambda(K^3_5)-10^{-3}$, where $K^{3-}_5$ is the $3$-graph obtained by removing one edge from $K^{3}_5$.
Hence, we may assume that $n\geq 6$. By our assumption, there exists a linear order $\mu$ on $[n]$ such that
$\forall i,j\in [n]$ whenever $i<_\mu j$ we have   $x_i\geq x_j$ and that $G$ is compressed relative to $\mu$. By relabelling if needed,
we may assume that $\mu$ is the natural order $1<2<\dots<n$. Then $x_1\geq x_2\geq \dots \geq x_n$.
By Fact \ref{dense}, $G$ covers pairs. So $i(n-1)n\in G$, for some $i<n-1$.
Since $G$ is compressed relative to the natural order, we have $1(n-1)n\in G$. Again, since $G$ is compressed relative to the natural order, this implies that
$\forall i,j$, where $2 \le i<j \le n$, $1ij\in G$.  Suppose that $G[\{2,\dots,n\}]$ contains an edge $e$. Since $n\geq 6$, $\exists i,j\in \{2,\ldots n\}$,
such that $i,j\notin e$. Now, $\{1ij, e\}$ forms a $2$-matching in $G$, contradicting $G$ being $M^3_2$-free.
Hence
 $G = \{1ij:2\le i < j \le n \} $.  Assume that $x_1=a$.  Since $\vec{y}=(\frac{x_2}{1-a},\dots, \frac{x_n}{1-a})$ is a feasible weight vector on $L_G(1)$,
 by  Theorem \ref{MStheo}
$$\lambda(G) = \lambda(G,\vec{x})=a(1-a)^2\lambda(L_G(1),\vec{y})< {1 \over 2}a(1-a)^2 \le {1 \over 4}\left[{2a+(1-a)+(1-a) \over 3}\right]^3 = {2 \over 27}
<\lambda(K_{5}^{3})-10^{-3}.$$
\end{proof}

\bigskip

We now extend Theorem \ref{theoHK} to determine (with stability) the maximum Lagrangian of a $3$-graph not containing a $t$-matching,
for all $t \ge 2$. Given an $r$-graph $G=(V,E)$ and $i\in V$, let $$I_G(i)=\{e \in G: i \in e\}.$$
\begin{theo}\label{Mt-free}
Let $t\ge 2$ be a positive integer. Let $G$ be an $M_t^3$-free $3$-graph with no isolated vertex and $G \neq K_{3t-1}^{3}$. Then there exists a positive real $c_1=c_1(t)$ such that $\lambda(G) \le \lambda(K_{3t-1}^{3})-c_1={1 \over 6}\left({[3t-1]_3 \over (3t-1)^3}-6c_1 \right)$.
\end{theo}
\begin{proof}
By Lemma \ref{compression-dense}, it suffices to assume that $G$
is dense and has an optimum weight vector $\vec{x}$ such that $G$ is $\vec{x}$-compressed.
Suppose $V(G)=[n]$. Let $K_{3t-1}^{3-}$ be the $3$-graph obtained by removing one edge from $K_{3t-1}^{3}$. If $n\le 3t-1$, then
since $G\neq K_{3t-1}$, $\lambda(G) \le \lambda(K_{3t-1}^{3-})$.  So, we may assume that $n\ge 3t$. We use induction on $t$, with Theorem \ref{theoHK} forming the basis step $t=2$. For the induction step, let $t\geq 3$.
By our assumption, there exists a linear order $\mu$ on $[n]$ such that
$\forall i,j\in [n]$,   $x_i\geq x_j$ if $i<_\mu j$  and that $G$ is compressed relative to $\mu$. By relabelling if needed,
we may assume that $\mu$ is the natural order $1<2<\dots<n$. Then $x_1\geq x_2\geq \dots \geq x_n$.
By Fact \ref{dense}, $G$ covers pairs. So $i(n-1)n\in G$, for some $i<n-1$. Since $G$ is compressed relative
to the natural order, this implies $1(n-1)n\in G$ and furthermore
\begin{equation} \label{IG}
I_G(1)=\{1ij:2\le i < j \le n\}.
\end{equation}
Suppose $x_{1}=a$. Then $0<a < 1$. Since $\vec{z}=(\frac{x_2}{1-a},\dots, \frac{x_n}{1-a})$ is
a feasible weight vector on $L_G(1)=K^2_{n-1}$. By Theorem \ref{MStheo}, we have
$$\lambda(I_G(1),\vec{x})=a\cdot\sum_{2\leq i<j\leq n} x_ix_j=a(1-a)^2 \lambda(L_G(1),\vec{z})
< {1\over 2}a(1-a)^2.$$
Let $F=G[\{2,3,\dots,n\}]$. Suppose $F$ contains a $(t-1)$-matching $M$. Since $n\geq 3t$, there exist
distinct vertices $i,j\in [n]\setminus (V(M)\cup \{1\})$. By \eqref{IG}, $1ij\in G$. Now, $M\cup \{1ij\}$
is a $t$-matching in $G$, contradicting $G$ being $M^3_t$-free.
Hence $F$ must be $M^3_{t-1}$-free. Note that $\vec{z}$ is a feasible weight vector on $F$. By
the induction hypothesis (by considering $F=K^3_{3t-4}$ or not), we have $\lambda(F,\vec{z})\leq \lambda(K_{3t-4}^{3})$.
Thus,

$$\lambda(F,\vec{x})=(1-a)^3\cdot \lambda(F,\vec{z}) \le(1-a)^3\lambda(F)\leq (1-a)^3\lambda(K_{3t-4}^{3})={3t-4 \choose 3}\left({1-a \over 3t-4}\right)^3.$$
Let $s=3t-4$ and $\mu ={s^2-3s+2 \over 6s^2}$. We have
\begin{eqnarray*}
\lambda(G)=
\lambda(G,\vec{x})&\le& \lambda(I_G(1),\vec{x})+\lambda(F,\vec{x}) \\
&<& {1\over 2}a(1-a)^2+{s \choose 3}\left({1-a \over s}\right)^3 \\
&=& {1\over 2}a(1-a)^2+{s^2-3s+2 \over 6s^2}(1-a)^3 \\
&=& (1-a)^2\left({1 \over 2}a+\mu(1-a)\right) \\
&=& (1-a)^2\left(\left({1 \over 2}-\mu \right)a+\mu \right) \\
&=&(1-a)(1-a) \left (2a+\frac{\mu}{\frac{1}{4}-\frac{1}{2}\mu}\right)\cdot \left(\frac{1}{4}-\frac{1}{2}\mu\right)\\
&\leq& \left[\frac{1}{3}\left(1-a+1-a+2a+\frac{\mu}{\frac{1}{4}-\frac{1}{2}\mu}\right)\right]^3\cdot \left(\frac{1}{4}-\frac{1}{2}\mu\right)\quad \mbox{(by the AM-GM inequality)}\\
&=& \frac{1}{54\left({1 \over 2}-\mu \right)^2}\\
&=&{2s^4 \over 3(2s^2+3s-2)^2}.
\end{eqnarray*}
Since $s=3t-4$, we have
$$ \lambda(K_{3t-1}^{3})={3t-1 \choose 3}\left({1 \over 3t-1}\right)^3=\binom{s+3}{3}\cdot \left(\frac{1}{s+3}\right)^3= {s^2+3s+2 \over 6(s+3)^2}.$$
Hence,
\begin{eqnarray*}
\lambda(G)- \lambda(K_{3t-1}^{3}) &\le& {2s^4 \over 3(2s^2+3s-2)^2}-{s^2+3s+2 \over 6(s+3)^2} \\
&=& {4s^4(s+3)^2-(2s^2+3s-2)^2(s^2+3s+2) \over 6(2s^2+3s-2)^2(s+3)^2} \\
&=& -{9s^4+15s^3-30s^2-12s+8 \over 6(2s^2+3s-2)^2(s+3)^2},
\end{eqnarray*}
which is negative for every $s\ge 2$. Let $$c_1=\min \left\{\lambda(K_{3t-1}^{3})-\lambda(K_{3t-1}^{3-}),  {9s^4+15s^3-30s^2-12s+8 \over 6(2s^2+3s-2)^2(s+3)^2} \right\}.$$
Then $\lambda(G) \le \lambda(K_{3t+2}^{3})-c_1$ and the proof is complete.
\end{proof}

\begin{coro}\label{coro}
$\pi_{\lambda}(M_t^3)=3! \lambda(K_{3t-1}^{3})={[3t-1]_3 \over (3t-1)^3}.$
\end{coro}
\begin{proof}
Since $K_{3t-1}^{3}$ is $M_t^3$-free, $\pi_{\lambda}(M_t^3)\ge 3! \lambda(K_{3t-1}^{3})$. On the other hand, by Theorem \ref{Mt-free}, $\pi_{\lambda}(M_t^3)\le 3! \lambda(K_{3t-1}^{3})$. Therefore, $\pi_{\lambda}(M_t^3)=3! \lambda(K_{3t-1}^{3})$.
\end{proof}


\subsection{ Tur\'an number of the extension of $M_t^3$}
The main result in this section is as follows.
\begin{theo}\label{t-matching-lagrangian}
Let $t\ge 2$ be an integer. Then $ex(n,H_{3t}^{M_t^3})= t_{3t-1}^{3}(n)$ for sufficiently large $n$. Moreover, if $n$ is sufficiently large and $G$ is an $H_{3t}^{M_t^3}$-free $3$-graph on $[n]$ with $|G|=t_{3t-1}^{3}(n)$, then $G=T_{3t-1}^{3}(n)$.
\end{theo}

To prove the theorem, we need several results from \cite{BIJ}. Similar results are obtained independently in \cite{NY2}.

\begin{defi}{\em (\cite{BIJ})
Let $m,r\ge 2$ be positive integers. Let $F$ be an $r$-graph that has at most $m+1$ vertices satisfying $\pi_{\lambda}(F)\le {[m]_r \over m^r} $. We say that ${\mathcal{K}}_{m+1}^F $ is {\em $m$-stable} if for every real $\varepsilon > 0$ there are a real $\delta > 0$ and an integer $n_1$ such that if $G$ is a ${\mathcal{K}}_{m+1}^F $-free $r$-graph with at least $n\ge n_1$ vertices and more than $({[m]_r \over m^r}-\delta){n \choose r} $ edges, then $G$ can be made $m$-partite by deleting at most $\varepsilon n$ vertices.}
\end{defi}
\begin{theo}{\rm (\cite{BIJ})}\label{BIJ-main}
Let $m,r\ge 2$ be positive integers. Let $F$ be an $r$-graph that either has at most $m$ vertices or has $m+1$ vertices one of which has degree $1$. Suppose either $\pi_{\lambda}(F)< {[m]_r \over m^r} $ or $\pi_{\lambda}(F)={[m]_r \over m^r} $ and ${\mathcal{K}}_{m+1}^F $ is $m$-stable. Then there exists a positive integer $n_2$ such that for all $n \ge n_2$ we have $ex(n,H_{m+1}^{F})= t_{m}^{r}(n)$ and the unique extremal $r$-graph is $T_{m}^{r}(n)$. \qed
\end{theo}

Given an $r$-graph $G$ and a real $\alpha$ with $0 < \alpha \le 1$, we say that $G$ is {\em $\alpha$-dense} if $G$ has minimum degree at least $\alpha {|V(G)|-1 \choose r-1}$. Let $i,j \in V(G)$, we say $i$ and $j$ are {\em nonadjacent} if $\{i,j\}$ is not contained in any edge of $G$. Given a set $U \subseteq V(G)$, we say $U$ is an {\em equivalence class} of $G$ if for every two vertices $u,v \in U$, $L_G(u)=L_G(v)$. Given two nonadjacent nonequivalent vertices $u,v \in V(G)$,  {\em symmetrizing} $u$ to $v$ refers to the operation of deleting all edges containing $u$ of $G$ and
adding all the edges $\{u\}\cup A, A\in L_G(v)$ to $G$.
We use the following algorithm from \cite{BIJ}, which was originated in \cite{Pikhurko}.

\begin{algorithm}\label{al1}{\em(Symmetrization and cleaning with threshold $\alpha$)}
\newline
\noindent{\bf Input:} {\em An $r$-graph $G$.}
\newline
\noindent{\bf Output:} {\em An $r$-graph $G^*$.}
\newline
\noindent{\bf Initiation:} {\em  Let $G_0=H_0=G.$ Set $i=0$.}
\newline
\noindent{\bf Iteration:}
{\em For each vertex $u$ in $H_i$, let $A_i(u)$ denote the equivalence class that $u$ is in. If either $H_i$ is empty or $H_i$ contains no two nonadjacent nonequivalent vertices, then let $G^*=H_i$ and terminate. Otherwise let $u,v$ be two nonadjacent nonequivalent vertices in $H_i$, where $d_{H_i}(u) \ge d_{H_i}(v)$. We symmetrize each vertex in $A_i(v)$ to $u$. Let $G_{i+1}$ denote the resulting graph.
If $G_{i+1}$ is $\alpha$-dense, then let $H_{i+1}=G_{i+1}$. Otherwise we let $L=G_{i+1}$ and repeat the following: let $z$ be any vertex of minimum degree in $L$. We redefine $L=L-z$ unless in forming $G_{i+1}$ from $H_i$ we symmetrized the equivalence class of some vertex $v$ in $H_i$ to some vertex in the equivalence class of $z$ in $H_i$. In that case, we redefine $L=L-v$ instead. We repeat the process until $L$ becomes either $\alpha$-dense or empty. Let $H_{i+1}=L$. We call the process of forming $H_{i+1}$ from $G_{i+1}$ ``cleaning". Let $Z_{i+1}$ denote the set of vertices removed, so that $H_{i+1}=G_{i+1}-Z_{i+1}$. By our definition, if $H_{i+1}$ is nonempty then it is $\alpha$-dense.
}
\end{algorithm}

\begin{theo}{\em (\cite{BIJ})}\label{theo6}
Let $m,r\ge 2$ be positive integers. Let $F$ be an $r$-graph that has at most $m$ vertices or has $m+1$ vertices one of which has degree $1$. There exists a real $\gamma_0=\gamma_0(m,r)>0$ such that for every positive real $\gamma <\gamma_0$, there exist a real $\delta > 0$ and an integer $n_0$ such that the following is true for all $n\ge n_0$. Let $G$ be an ${\mathcal{K}}_{m+1}^F$-free $r$-graph on $[n]$ with more than $({[m]_r \over m^r}-\delta){n \choose r} $ edges. Let $G^*$ be the final $r$-graph produced by Algorithm \ref{al1} with threshold ${[m]_r \over m^r}-\gamma$. Then $|V(G^*)|\ge (1-\gamma)n$ and $G^*$ is $({[m]_r \over m^r}-\gamma)$-dense. Furthermore, if there is a set $W \subseteq V(G^*)$ with $|W|\ge (1-\gamma_0)|V(G^*)|$ such that $W$ is the union of a collection of at most $m$ equivalence classes of $G^*$, then $G[W]$ is $m$-partite. \qed
\end{theo}

The following corollary is implicit in \cite{BIJ} and \cite{NY2}.
\begin{coro}\label{generalcoro}
Let $m,r\ge 2$ be positive integers. Let $F$ be an $r$-graph that has at most $m+1$ vertices with a vertex of degree $1$ and $\pi_{\lambda}(F)\le {[m]_r \over m^r} $.
Suppose there is a constant $c>0$ such that
for every $F$-free $r$-graph $L$ with no isolated vertex and $L\neq K_{m}^r$, $\lambda(L)\le \lambda(K_{m}^{r})-c$. Then $\cK_{m+1}^{F}$ is $m$-stable.
\end{coro}
\begin{proof} Let $\varepsilon >0$ be given. Let $\delta, n_0$ be the constants guaranteed by Theorem \ref{theo6}. We can assume that $\delta$ is small enough and $n_0$ is large enough. Let $\gamma >0$ satisfy $\gamma <\varepsilon$ and $ \delta+r\gamma < c$. Let $G$ be a ${\mathcal{K}}_{m+1}^{F}$-free $r$-graph on $n>n_0$ vertices with more than $({[m]_r \over m^r}-\delta){n \choose r} $ edges. Let $G^*$ be the final $r$-graph produced by applying Algorithm \ref{al1} to $G$ with threshold
${[m]_r \over m^r}-\gamma$. By Algorithm \ref{al1}, if $S$ consists of one vertex from each equivalence class of $G^*$, then $G^*[S]$ covers pairs and $G^*$ is a blowup of $G^*[S]$.

First, suppose that $|S| \ge m+1$. If $F \subseteq G^*[S]$, then since $G^*[S]$ covers pairs we can find a member of ${\mathcal{K}}_{m+1}^{F}$ in $G^*[S]$ by using any $(m+1)$-set that contains a copy of $F$ as the core, contradicting $G^*$ being ${\mathcal{K}}_{m+1}^{F}$-free. So $G^*[S]$ is $F$-free.  Since $|S|\geq m+1$ and $G^*[S]$ covers pairs, clearly $G^*[S]\neq K^r_m$. Also, $G^*[S]$ has no isolated vertex. Hence, by our assumption, $\lambda(G^*[S])\leq \frac{1}{r!} \frac{[m]_r}{m^r}-c$. By Proposition \ref{lag-bound}, we have
\begin{equation} \label{G^*-upper}
|G^*|\leq \lambda(G^*[S])n^r\leq (\frac{1}{r!}\frac{[m]_r}{m^r}-c)n^r< (\frac{[m]_r}{m^r}-c)\frac{n^r}{r!}.
\end{equation}
Now, during the process of obtaining $G^*$ from $G$, symmetrization never decreases the number of edges.
Since at most $\gamma n$ vertices are deleted in the process (see Theorem \ref{theo6}),
$$|G^*|>|G|-\gamma n {n-1 \choose r-1} \ge \left({[m]_r \over m^r}-\delta-r\gamma \right){n \choose r} > \left({[m]_r \over m^r}-c\right){n^r \over r!},$$
contradicting \eqref{G^*-upper}. So $|S|\le m$. Hence, $W=V(G^*)$ is the union of at most $m$ equivalence classes of $G^*$. By Theorem \ref{theo6},
$|W|\geq (1-\gamma) n$ and $G[W]$ is $m$-partite. Hence, $G$ can be made $m$-partite by deleting at most $\gamma n<\varepsilon n$ vertices. Thus,
$\cK_{m+1}^F$ is $m$-stable.
\end{proof}

\medskip

\noindent{\bf Proof of Theorem \ref{t-matching-lagrangian}.} By Theorem \ref{Mt-free} and Corollary \ref{coro}, $M_t^3$ satisfies the conditions of Corollary \ref{generalcoro}. So,
 $\cK_{3t}^{M_t^3}$ is $(3t-1)$-stable. The theorem then follows from Theorem \ref{BIJ-main}.
\qed
\medskip


\section{Local Lagrangians of $M^r_t$-free $r$-graphs and Lagrangians of $L^r_t$-free $r$-graphs and related Tur\'an numbers}

In this section we consider a local version of Lagrangians of $M^r_t$-free graphs for $r=2,3$.
This will then be used to determine the Lagrangian density of  a linear star $L^r_t$ for $r=3,4$.
Let $0<b<1$ be a real. Given an $r$-graph $G$ on $[n]$, a feasible weight vector $\vec{x}=(x_1,\dots, x_n)$ is called a \emph{$b$-bounded feasible weight vector}
on $G$ if $ \forall i \in [n]$, $x_i\leq b$.  If $G$ has a $b$-bounded feasible
weight vector, then we define the \emph{$b$-bounded Lagrangian} of $G$ as

\begin{equation} \label{b-bounded-definition}
\lambda_b(G)=\max\{\lambda(G,\vec{x}): \vec{x} \mbox{ is a $b$-bounded feasible weight vector on $G$}\}.
\end{equation}
If $G$ does not have any $b$-bounded feasible weight vector, then we define $\lambda_b(G)=0$.
A feasible $b$-bounded weight vector $\vec{x}$ on $G$ such that $\lambda(G,\vec{x})=\lambda_b(G)$
is called an \emph{optimum $b$-bounded weight vector} on $G$.
We now consider $\lambda_b(G)$ over $M^r_t$-free $r$-graphs for $r=2,3$ for appropriate values of $b$.
For such a study, first we reduce the problem to the case where the $r$-graph in consideration is compressed and
there exists an optimum $b$-bounded weight vector with some additioal properties.

\begin{lemma} \label{b-bound-facts}
Let $0<b<1$ be a real. Let $r$, $t\geq 2$ be integers. Let $\cF$ be the family of all $M^r_t$-free $r$-graphs.
There exists $G\in \cF$ and an optimum $b$-bounded weight vector $\vec{x}$ on $G$ such that
\begin{enumerate}
\item $\lambda(G,\vec{x})=\lambda_b(G)=\max\{\lambda_b(H): H\in \cF\}$.
\item $G$ is $\vec{x}$-compressed.
\item All vertices of $G$ have positive weight under $\vec{x}$.
\item If $u,v$ are any two vertices in $G$ with weight less than $b$ under $\vec{x}$ then $\{u,v\}$ is covered in $G$.
\end{enumerate}
\end{lemma}
\begin{proof}
Clearly, $\cF$ is closed under taking subgraphs.
Let $\lambda^*=\max\{\lambda_b(H): H\in \cF\}$. Among all $r$-graphs $H\in \cF$ with $\lambda_b(H)=\lambda^*$,
let $G$ be the one with the fewest possible vertices. Let $\vec{x}$ be an optimum $b$-bounded weight vector on $G$
that has the maximum number of $b$-components. By Lemma \ref{compression-lemma} {\rm (a)}, we may assume that
$G$ is $\vec{x}$-compressed (or else we could replace $G$ with one that is $\vec{x}$-compressed).
If some vertex in $G$ has $0$ weight under $\vec{x}$ then deleting that vertex would give us a graph $G' \in \cF$
with $\lambda_b(G')=\lambda^*$ and having fewer vertices than $G$, contradicting our choice of $G$. Hence,
all vertices in $G$ have positive weights under $\vec{x}$.
Now, suppose $u,v$ are two vertices with weight less than $b$ under $\vec{x}$.
Suppose that no edge of $G$ contains both $u$ and $v$. Without loss of generality suppose that $\lambda(L_G(u),\vec{x})\geq
\lambda(L_G(v),\vec{x})$. If we decrease the weight of $v$ and increase the weight of $u$ by the same amount,
the total weight does not decrease. Hence, we can obtain an optimum $b$-bounded weight vector on $G$
that either has more $b$-components than $\vec{x}$ or has weight $0$ on $v$. In the former, we get
a contradiction to our choice of $\vec{x}$. In the latter case, we get a contradiction to our choice of $G$.
Hence there must be some edge in $G$ containing both $u$ and $v$.
\end{proof}

For the purpose of studying $L^r_t$-free graphs, we will also need the following short lemma.

\begin{lemma} \label{K-free}
Let $r,t\geq 2$. Let $G$ be an $L^r_t$-free $r$-graph with at least $t(r-1)+1$ vertices and $G$ covers pairs. Let $x\in V(G)$.
Then $L(x)$ is $K^{r-1}_{t(r-1)-1}$-free. In particular, $G$ is $K^r_{t(r-1)}$-free.
\end{lemma}
\begin{proof}
Suppose for contradiction that $L(x)$  contains a copy $K$ of $K^{r-1}_{t(r-1)-1}$. By our assumption, $\exists$ $y\in V(G)\setminus
(V(K)\cup \{x\})$. Since $G$ covers pairs, there exists $e\in G$ that contains $x$ and $y$.
Now we can find a copy of $L^r_{t-1}$ using a $(t-1)$-matching  in $K$ containing $x$ that are disjoint from $e\setminus \{x,y\}$,
which together with $e$ form a copy of $L^r_t$ in $G$, a contradiction.
\end{proof}

\subsection{Local Lagrangians of $M^2_t$-free graphs and Lagrangians of $L^3_t$-free $3$-graphs and related Tur\'an numbers}

We start the subsection by developing some structural properties of $M^2_t$-free left-compressed graphs.
Let $n,t$ be positive integers, where $t\geq 2$ and $n\geq 2t$.
For each $\ell\in [t-1]\cup\{0\}$, define $$F_{t,\ell}(n)=\binom{[2t-1-\ell]}{2} \cup \left\{ab: a\in\{1,\dots, \ell\}, \,\, b\in \{2t-\ell,\dots,n\}\right\}.$$
Note that
$F_{t,\ell}(n)$ is $M_t^2$-free for each $\ell\in [t-1]\cup\{0\}$.

\begin{lemma} \label{Mt-free-2-graph}
Let $n,t$ be positive integers, where $t\geq 2$ and $n\geq 2t$.
Let $G$ be an $M^2_t$-free $2$-graph on $[n]$ that is left-compressed relative to the natural order.
Then $G\subseteq F_{t,\ell}(n)$ for some $\ell\in [t-1]\cup\{0\}$.
\end{lemma}
\begin{proof}
For each $i\in [t]$, let $N_i=\{j\in [n]: j>i, ij\in G\}$. Since $G$ is left-compressed relative to
the natural order on $[n]$, we have either $N_i=\emptyset$ or $N_i=\{i+1,i+2\,\dots, m_i\}$ for some $m_i>i$. Furthermore,
$N_1\supseteq N_2\supseteq\dots\supseteq N_t$. For convenience, we define $m_i=1$ for those
$i\in [t]$ with $N_i=\emptyset$. Then $\{m_1,\dots, m_t\}$ is non-increasing. Let $h$ be the largest
$i\in [t]$ such that $m_i\leq 2t-i$. Note that $h$ exists; otherwise $\{i(2t+1-i): i\in [t]\}$ is
a $t$-matching in $G$, a contradiction. Let $\ell=h-1$. Then $\ell\in [t-1]\cup\{0\}$.
By our assumption, there is no edge from $[\ell+1,n]$ to $[2t-\ell,n]$.
So $G\subseteq F_{t,\ell}(n)$.
\end{proof}

\medskip

\begin{lemma} \label{local-t-free}
Let $n,t$ be positive integers, where $t\geq 2$ and $n\geq 2t$. Let $b$ be a real such that
$0<b\le \frac{1}{t}$. For each $\ell\in [t-1]$, we have $\lambda_b(F_{t,\ell}(n))\le
\binom{2t-1-2\ell}{2}b^2+\ell b-\frac{\ell^2+\ell}{2} b^2$.
\end{lemma}
\begin{proof}
Let $\ell \in [t-1]$. Let $\vec{x}=(x_1,\dots, x_n)$ be a $b$-bounded feasible vector on $F_{t,\ell}(n)$
such that $\lambda(F_{t,\ell}(n), \vec{x})=\lambda_b(F_{t,\ell}(n))$. Using Fact \ref{symmetry} (note that
any new weight vector produced by Fact \ref{symmetry} based on $\vec{x}$ is also $b$-bounded),
we may assume that
 $x_1=\dots=x_\ell$, $x_{\ell+1}=\dots=x_{2t-1-\ell}$ and
$x_{2t-\ell}=\dots=x_n$. Let $a=x_1$, $c=x_{\ell+1}$, and $d=x_{2t-\ell}+\dots+x_n=1-\ell a- (2t-1-2\ell)c$.
We have
\begin{eqnarray*}
\lambda(F_{t,\ell}(n),\vec{x})&=&\binom{\ell}{2}a^2+\binom{2t-1-2\ell}{2}c^2+(2t-1-2\ell)\ell ac+\ell a[1-\ell a-(2t-1-2\ell)c] \\
&=&\binom{\ell}{2}a^2+\binom{2t-1-2\ell}{2}c^2+\ell a(1-\ell a)\\
&=&\binom{2t-1-2\ell}{2}c^2+\ell a-\frac{\ell^2+\ell}{2} a^2\\
&\leq& \binom{2t-1-2\ell}{2}b^2+\ell b-\frac{\ell^2+\ell}{2} b^2,\\
\end{eqnarray*}
where we used the fact that $f(x)=\ell x -\frac{\ell^2+\ell}{2} x^2$ is increasing on  $(-\infty, \frac{1}{\ell+1})$
and that $a,c\leq b\leq \frac{1}{\ell+1}$.
\end{proof}

\begin{theo}\label{L^3_t}
Let $t\ge 2$ be an integer. If $G$ is an $L_t^3$-free $3$-graph, then $\lambda(G)\le \lambda(K_{2t}^{3})$. Furthermore, there is  $c_2=c_2(t)>0$  such that if $G$
is an $L_t^3$-free $3$-graph that covers pairs and $G \neq K_{2t}^3$ then
$\lambda(G)\le \lambda(K_{2t}^{3})-c_2= {(2t-1)(t-1) \over 12t^2}-c_2$.
\end{theo}
\begin{proof}
It suffices to assume that $G$ is dense (otherwise we consider an appropriate subgraph). So $G$ covers pairs.
In this set up, it  suffices to prove the second statement. So assume that $G$ covers pairs
and $G \neq K_{2t}^3$.  Suppose $V(G)=[n]\cup\{0\}$.
If $n<2t$, then $\lambda(G)\leq \lambda(K_{2t}^{3-})\leq \lambda(K_{2t}^{3})-c_2$, by choosing $c_2$ to be small enough, where $K_{2t}^{3-}$ denotes $K^3_{2t}$ minus
an edge. Hence, we may assume that $n\geq 2t$. Let $\vec{x}=(x_0,x_1,\dots,x_n)$ be an optimum weight vector on $G$.
Let $a=\max \{x_i:i \in V(G)\}$. By relabeling if needed, we may assume that $x_0=a$.
By Fact \ref{fact2}, $\lambda(L(0),\vec{x})={\partial \lambda (G,\vec{x}) \over \partial x_0}=3\lambda(G)$, so it suffices to show that $\lambda(L(0),\vec{x})\le {(2t-1)(t-1) \over 4t^2}-3c_2$, for some sufficiently small positive real $c_2$.

Since $G$ is $L_t^3$-free, $L(0)$ is  $M_t^2$-free. Since $G$ covers pairs and $n\geq 2t$, by Lemma \ref{K-free},
$K^2_{2t-1}\not\subseteq L(0)$. We may view $L(0)$ as a $2$-graph on $[n]$. Let $\vec{y}=(\frac{x_1}{1-a}, \dots, \frac{x_n}{1-a})$. Then
$\vec{y}$ is a feasible weight vector on $L(0)$. Furthermore, it is $\frac{a}{1-a}$-bounded.
We consider two cases.

\medskip

{\bf Case 1.} $a \ge \frac{1}{2t}$.

\medskip

Since $L(0)$ is $K^2_{2t-1}$-free, by Theorem \ref{MStheo}, $\lambda(L(0))\leq \frac{1}{2} (1-\frac{1}{2t-2})$.
Hence, for sufficiently small $c_2>0$,
$$\lambda(L(0),\vec{x})=(1-a)^2 \lambda(L(0),\vec{y})\leq (1-a)^2\lambda(L(0))\leq (\frac{2t-1}{2t})^2\frac{1}{2}\frac{2t-3}{2t-2}
<\frac{(2t-1)(t-1)}{4t^2}-3c_2.$$

\medskip

{\bf Case 2.} $a < {1 \over 2t}$.

\medskip
Let $b=\frac{a}{1-a}$. Then $b< \frac{1}{2t-1}\leq \frac{1}{t}$.
By Lemma \ref{compression-lemma} {\rm (a)},
there exists a $M_t^2$-free $2$-graph $H$ on $[n]$ such that $\lambda(H,\vec{y})\geq \lambda(L(0),\vec{y})$
and such that $H$ is $\vec{y}$-compressed. Also, since $L(0)$ is $K^2_{2t-1}$-free, by Lemma \ref{t-free-2-uniform},
$H$ is also $K^2_{2t-1}$-free.
By relabeling if needed, we may assume that $y_1\geq \dots \geq y_n$ and that $H$ is left-compressed
relative to the natural order on $[n]$. By Lemma \ref{Mt-free-2-graph}, $H\subseteq F_{t,\ell}(n)$ for
some $\ell\in [t-1]\cup \{0\}$.  First, assume that $\ell\in [t-1]$.

Since  $\vec{y}$ is a $b$-bounded feasible weight vector on $[n]$,
by Lemma \ref{local-t-free},  we have

\begin{eqnarray*}
\lambda(L(0),\vec{x})&=& (1-a)^2 \lambda(L(0),\vec{y})\leq \lambda(H,\vec{y})\leq \lambda(F_{t,\ell}(n),\vec{y})\\
&\leq& (1-a)^2\left[ \binom{2t-1-2\ell}{2}(\frac{a}{1-a})^2+\ell \frac{a}{1-a}-\frac{\ell^2+\ell}{2} (\frac{a}{1-a})^2\right ]\\
&=&\binom{2t-1-2\ell}{2}a^2+\ell a(1-a) -\frac{\ell^2+\ell}{2}a^2.\\
\end{eqnarray*}
Since $f(x)=\ell x(1-x)-\frac{\ell^2+\ell}{2} x^2$ increases on $(-\infty, \frac{1}{\ell+3})$ and $a<\frac{1}{2t}\leq \frac{1}{\ell+3}$, we
have
\begin{eqnarray*}
\lambda(L(0),\vec{x})&\leq& \binom{2t-1-2\ell}{2}\left(\frac{1}{2t}\right)^2 +\ell \frac{1}{2t}\left(1-\frac{1}{2t}\right)
-\frac{\ell^2+\ell}{2}\left(\frac{1}{2t}\right)^2\\
&=&\binom{2t-1-2\ell}{2}\left(\frac{1}{2t}\right)^2 +\ell \frac{1}{2t}\left(1-\frac{2}{2t}\right)
-\binom{\ell}{2}\left(\frac{1}{2t}\right)^2\\
&=&\lambda(F_{t,\ell}(2t-1), \vec{z}),
\end{eqnarray*}
where $z$ is a weight vector on $[2t-1]$ with $z=(\frac{1}{2t},\dots,\frac{1}{2t})$.
Since $\ell\geq 1$, $F_{t,\ell}(2t-1)\subseteq K^{2-}_{2t-1}$.
Hence,
\begin{equation} \label{Ft}
\lambda(L(0),\vec{x})\leq \lambda(K^{2-}_{2t-1},\vec{z})\leq \frac{(2t-1)(t-1)}{4t^2}-3c_2,
\end{equation}
for sufficiently small $c_2>0$.

Finally, suppose $\ell=0$. Note that $F_{t,0}(n)$ consists of a copy of $K^2_{2t-1}$ and some isolated vertices.
Since $H\subseteq F_{t,0}(n)$ and $H$ is $K^2_{2t-1}$-free, we have $\lambda(L(0),\vec{x})\leq
\lambda(H,\vec{x})\leq \lambda(H)\leq \lambda(K^{2-}_{2t-1})$. Hence \eqref{Ft} still holds for sufficiently small $c_2>0$.
 This completes our proof.
\end{proof}

\begin{coro}\label{CoroL^3_t}
$\pi_{\lambda}(L^3_t)=3! \lambda(K_{2t}^{3})={[2t]_3 \over (2t)^3}.$
\end{coro}

Applying Theorem \ref{L^3_t}, Corollary \ref{CoroL^3_t}, Corollary \ref{generalcoro} and Theorem \ref{BIJ-main}, we have

\begin{theo}\label{theoL^3_t}
Let $t\ge 2$ be an integer. Then $ex(n,H_{2t+1}^{L^3_t})= t_{2t}^{3}(n)$ for sufficiently large $n$. Moreover, if $n$ is sufficiently large and $G$ is an $H_{2t+1}^{L^3_t}$-free $3$-graph on $n$ vertices with $|G|=t_{2t}^{3}(n)$ then $G=T_{2t}^{3}(n)$. \qed
\end{theo}

Theorem \ref{theoL^3_t} is part of a more general theorem obtained in \cite{BIJ} and \cite{NY2}.
However, the method we used in this section is self-contained and is very different from those used in \cite{BIJ} and \cite{NY2}.


\subsection{Local Lagrangians of $M^3_t$-free $3$-graphs and Lagrangians of $L^4_t$-free $4$-graphs and related Tur\'an numbers}

Next, we consider local Lagrangians of $M^3_t$-free $3$-graphs. First, we focus on the $t=2$ case. As before, we first develop
some structural properties of $M^3_2$-free $3$-graphs.
Given a $3$-graph $G$ on $[n]$, let $L^+(1)$ and $L^+(2)$ denote the links of $1,2$ of $G$ in $[3,n]$ respectively, i.e.
$$L^+(i)=\{A \subseteq [3,n]:A \cup \{i\} \in G\}$$
for $i=1,2$.
We say a set $S\subseteq V(G)$ is a  \emph{vertex cover} of $G$ if for every edge $e$ of $G$, $e \cap S\neq \emptyset$.

\begin{lemma}\label{M32-free}
Let $n\geq 6$ be an integer.
Let $G$ be an $M^3_2$-free $3$-graph on $[n]$ with no isolated vertex that is left-compressed relative to the natural order on $[n]$.
Then

{\em(a)}  $\forall i\in [3,n], 12i\in G$,

{\em(b)}  $\{1,2\}$ is a vertex cover of $G$, and

{\em(c)}  $L^+(2)$ is $M_2^2$-free. Thus,  if $L^+(2)\neq \emptyset$ then $L^+(2)$ is either a triangle or a star.
\end{lemma}
\begin{proof}
By our assumption, for some $i<j<n$, $ijn\in G$. Since $G$ is left-compressed relative to the natural order on $[n]$,
we have $12n\in G$. Since $G$ is left-compressed, this further implies that $12i\in G$ for every $i\in [3,n]$.
If $G$ contains an edge $e$ not containing $1$ or $2$, then $\{12i, e\}$ would form a $2$-matching in $G$, for some
$i\in [n], i\notin e$ and $i\neq 1,2$, contradicting $G$ being $M^3_2$-free. Hence $\{1,2\}$ is a vertex cover of $G$.
Finally, since $G$ is left-compressed, $L^+(2)\subseteq L^+(1)$. If $L^+(2)$ contains a $2$-matching, then we would
obtain a $2$-matching in $G$, a contradiction. So $L^+(2)$ is intersecting and must be either a star or a triangle.
\end{proof}

\medskip

Lemma \ref{M32-free} allows us to describe all left-compressed $M^3_2$-free $3$-graphs on $[n]$.

\begin{defi} \label{Gi-definitions}
For all integers $n\geq 5$, let
\begin{eqnarray*}
G_0(n)&=&\{1ij: 2\leq i<j\leq n\},\\
G_1(n)&=&\{12i:3 \le i \le n\} \cup \{134, 135, 145, 234, 235, 245\},\\
G_2(n)&=&\binom{[4]}{3}\cup \{12i, 13i, 14i: 5\leq i\leq n\}, \\
G_3(n)&=&\{12i: 3 \le i \le n\} \cup \{13i: 4 \le i \le n\} \cup \{ 234, 235, 145\},\\
G_4(n) &=&\{12i:3 \le i \le n\} \cup \{13i:4 \le i \le n\} \cup \{23i:4 \le i \le n\}.\\
\end{eqnarray*}
\end{defi}

\begin{lemma} \label{structure}
Let $n\geq 6$ be an integer.
Let $G$ be an $M^3_2$-free $3$-graph on $[n]$ that is left-compressed relative to the natural order on $[n]$.
Then $G$ is a subgraph of  one of $G_0(n), G_1(n), G_2(n), G_3(n), G_4(n)$  given in Definition \ref{Gi-definitions}.
\end{lemma}
\begin{proof} By Lemma \ref{M32-free}, $\{1,2\}$ is a vertex cover of $G$ and
$L^+(2)$ is either empty, or a triangle
or a star. We now consider three cases.

\medskip

{\bf Case 1.} $L^+(2)=\emptyset$.

Since $G$ is left-compressed, $L^+(i)=\emptyset$ for all $i\geq 2$. Hence
$G\subseteq G_0(n)=\{1ij: 2\leq i<j\leq n\}$.

\medskip

{\bf Case 2.} $L^+(2)$ is a triangle.

Since $G$ is left-compressed, we have $L^+(2)=\{34, 35, 45\}$ and $L^+(1)\supseteq L^+(2)$.
Since $G$ contains no $2$-matching, we must have $L^+(1)=L^+(2)=\{34,35, 45\}$.
Hence $$G\subseteq G_1(n)=\{12i:3 \le i \le n\} \cup \{134, 135, 145, 234, 235, 245\}.$$

{\bf Case 3.} $L^+(2)$ is a star.

Since $G$ is left-compressed, we have $L^+(2)=\{34, 35, \ldots, 3p\}$ for some $4 \le p \le n$.
Since $G$ contains no $2$-matching, every member of $L(1\setminus 2)$ must contain either $3$ or $4$.
Further, if $p\geq 6$ then every member of $L(1\setminus 2)$ must contain $3$.

If $p=4$, then
$$G\subseteq G_2(n)=\binom{[4]}{3}\cup \{12i, 13i, 14i: 5\leq i\leq n\} .$$

If $p=5$, then
$$G\subseteq G_3(n)=\{12i: 3 \le i \le n\} \cup \{13i: 4 \le i \le n\} \cup \{ 234, 235, 145\}.$$

If $p\geq 6$, then
$$G\subseteq G_4(n)=\{12i:3 \le i \le n\} \cup \{13i:4 \le i \le n\} \cup \{23i:4 \le i \le n\}.$$
\end{proof}

Let us recall the definition of the $b$-bounded Lagrangian $\lambda_b(G)$ of $G$, given in \eqref{b-bounded-definition}.
\begin{lemma}\label{star-b-bound}
Let $b$ be a real with $0<b\leq \frac{1}{3}$. Let $G$ be a $3$-uniform star. Then $\lambda_b(G)\leq \frac{1}{2}b(1-b)^2$.
\end{lemma}
\begin{proof}
Suppose $V(G)=[n]$. Without loss of generality suppose vertex $1$ is the center of the star. Let $\vec{x}$ be a $b$-bounded
feasible vector on $G$ with $\lambda(G,\vec{x})=\lambda_b(G)$. Let $a=x_1$. Then $a\leq b$. Note that
$(\frac{x_2}{1-a},\dots, \frac{x_n}{1-a})$ is a feasible weight vector on $L_G(1)$.
By Theorem \ref{MStheo},
$\lambda(G,\vec{x})\leq a\cdot \frac{1}{2}(1-a)^2\leq \frac{1}{2}b(1-b)^2$, where the last inequality follows from
the fact that the function $\frac{1}{2}x(1-x)^2$ increases on $[0,\frac{1}{3}]$ and that $0<a\leq b\leq \frac{1}{3}$.
\end{proof}

\begin{lemma} \label{M^3_2-local-bound}
Let $G$ be an $M^3_2$-free $3$-graph. For $0<b\leq \frac{1}{5}$, we have
$$\lambda_b(G)\leq \max\{\frac{1}{2} b(1-b)^2, b^2+4b^3\}.$$
Furthermore, if $0<b\le \frac{1}{7}$ then $\lambda_b(G)\leq \frac{1}{2} b(1-b)^2$.
\end{lemma}
\begin{proof}
Suppose $V(G)=[n]$.
By Lemma \ref{b-bound-facts}, we may assume that $G$ has an optimum $b$-bounded weight vector $\vec{x}$
such that $G$ is $\vec{x}$-compressed, all vertices of $G$ have positive weights under $\vec{x}$, and such that
all pairs of vertices of weight less than $b$ are covered in $G$.
By relabeling the vertices of $G$ if needed
we may assume that $x_1\geq\ldots \geq x_{n}$ and that $G$ is left-compressed relative to the natural order on $[n]$.

\medskip

{\bf Case 1.} $x_{n-1}<b$.

\medskip

In this case we have $x_{n-1}, x_n<b$.
By our assumption, $\{n-1,n\}$ is covered in $G$. Since $G$ is left-compressed,
this implies that $\forall 2\leq i<j\leq n, 1ij \in G$.
If there is an edge of $G'$ in $\{2,\ldots, n\}$ then since $G$ is left-compressed, we have $234\in G$.
But then $234, 156$ forms a $M_2^3$ in $G'$, contradiction. Hence $G'\subseteq G_0(n)=\{1ij: 2\leq i<j\leq n\}$.
By Lemma \ref{star-b-bound}, $\lambda_b(G')\leq \lambda_b(G_0(n))\leq \frac{1}{2}b(1-b)^2$.

\medskip

{\bf Case 2.} $x_{n-1}=b$.

\medskip

In this case  we have $x_1=x_2=\cdots=x_{n-1}=b$, $x_n \le b$.
By Lemma \ref{structure}, $G'\subseteq G_i$ for some $i=0,1,2,3,4$.
Since $\lambda_b(G_0(n))\leq \frac{1}{2}b(1-b)^2$,
we may assume that $G'\subseteq G_i(n)$ for some $i\in [4]$.
Since $G_1(n)=\{12i:3 \le i \le n\} \cup \{134, 135, 145, 234, 235, 245\}$,
$$\lambda(G_1(n),\vec{x}) \le 6b^3+b^2(1-2b)=b^2+4b^3.$$

Since $G_2(n)={[4]\choose 3}\cup \{12i, 13i, 14i: 5\leq i\leq n\}$,
$$\lambda(G_2(n),\vec{x})\le 4b^3+3b^2(1-4b)=3b^2-8b^3.$$

Since $G_3(n)=\{12i: 3 \le i \le n\} \cup \{13i: 4 \le i \le n\} \cup \{ 234, 235, 145\}$,
$$\lambda(G_3(n),\vec{x}) \le b^2(1-2b)+b^2(1-3b)+3b^3=2b^2-2b^3.$$

Since $G_4(n)=\{12i:3 \le i \le n\} \cup \{13i:4 \le i \le n\} \cup \{23i:4 \le i \le n\}$,
$$\lambda(G_4(n),\vec{x})\le b^3 + 3b^2(1-3b)=3b^2-8b^3.$$

So $$\lambda(G',\vec{x})\leq \max\{\frac{1}{2}b(1-b)^2, b^2+4b^3, 3b^2-8b^3, 2b^2-2b^3\}=\max\{\frac{1}{2} b(1-b)^2, b^2+4b^3, 3b^2-8b^3\}.$$
Note that $\frac{1}{2}b(1-b)^2-(3b^2-8b^3)\geq 0$ on $[0,\infty)$. Also, $\frac{1}{2}b(1-b)^2-(b^2+4b^3)\geq 0$ on $[0,\frac{1}{7}]$.
The conclusion follows.
\end{proof}

\medskip

Next, we establish an upper bound on $\lambda_b(G)$ for $M^3_t$-free graphs $G$, where $t\geq 3$.
We need the following lemma of Frankl.

\begin{lemma} \label{frankl-general} {\rm \cite{frankl-survey}}
If $G$ is an $n$-vertex $r$-graph with matching number $s$ then
$|G|\leq s\binom{n-1}{r-1}$. \qed
\end{lemma}

\begin{lemma} \label{sub-structures}
Let $n,r,t$ be positive integers, where $r,t\geq 2$, $n\geq tr$.
Let $G$ be an $M^r_t$-free graph on $[n]$ that is left-compressed relative to the natural order.
Then $L_G(n)$ is $M^{r-1}_t$-free. Furthermore, if  $r=3$ and $\{n-1,n\}$ is covered then
$G[\{2,\ldots, n\}]$ is $M^3_{t-1}$-free.
\end{lemma}
\begin{proof}
Suppose for contradiction that $M=\{f_1,\dots, f_t\}$ is
a $t$-matching in $L_G(n)$. Together they cover $t(r-1)$ vertices in $[n-1]$. Since $n\geq tr$,
there exist distinct vertices $v_1,\ldots, v_{t-1}\in [n-1]$ that are not covered by $M$.
Since $G$ is left-compressed, $f_1\cup \{v_1\},\dots, f_{t-1}\cup \{v_{t-1}\}\in G$, which together
with $f_t\cup \{n\}$, form a $t$-matching in $G$, a contradiction.

Next, suppose $r=3$ and $\{n-1,n\}$ is covered.
Since $G$ is left-compressed we  have $\forall 2\leq i<j\leq n, 1ij\in G$.
Suppose $G[\{2,\dots, n\}]$ contains $(t-1)$-matching $M$. Then since $n\geq 3t$, $[n]\setminus \{1\}$ contains
two vertices $j,\ell$ not covered by $M$. Now, $M\cup \{1j\ell\}$ is a $t$-matching in $G$, a contradiction.
\end{proof}

\begin{lemma} \label{almost-all-b}
Let $t\geq 3$.
Let $G$ be an $M^3_t$-free $3$-graph. Let $0<b<\frac{1}{3t-1}$.
Let $\vec{x}$ be a $b$-bounded feasible weight vector on $G$ such that all but one of
the components of $\vec{x}$ are $b$. Then
$$\lambda(G,\vec{x})\leq
{t-1 \over 2}b(1-3b+4b^2).$$
\end{lemma}
\begin{proof}
Suppose $V(G)=[n]$.  Note that $n\geq 3t$.
By Lemma \ref{compression-lemma} {\rm (a)}, we may assume that $G$ is $\vec{x}$-compressed.
By relabeling the vertices of $G$ if needed we may assume that $x_1\geq\ldots \geq x_{n}$ and that $G$ is left-compressed relative to the natural order on $[n]$.
By our assumption, $x_1=\dots=x_{n-1}=b$. Suppose $x_{n}=\alpha b$, where $0<\alpha\leq 1$.
By Lemma \ref{sub-structures}, $L(n)$ is $M^{2}_t$-free. Hence by Lemma \ref{frankl-general}, $|L(n)|\leq (t-1)(n-1)$.
Let $G'$ denote the set of edges of $G$ not containing $n$. Since $G'$ is $M^3_t$-free, by
Lemma \ref{frankl-general}, $|G'|\leq (t-1)\binom{n-2}{2}$. Hence the contribution to $\lambda(G,\vec{x})$  of edges in $G$ containing $n$ or not containing $n$ are at most $(t-1)(n-1)b^2\cdot \alpha b$ and $(t-1)\binom{n-2}{2}b^3$ respectively.
Note that $(n-1)b+\alpha b=1$. Also, on $[0,1]$ we have
$\alpha^2-3\alpha+\frac{1}{4}\geq -\frac{7}{4}$. Hence

\begin{eqnarray*}\label{tcase1}
\lambda(G,\vec{x})&\leq& (t-1)\binom{n-2}{2}\cdot b^3+(t-1)\alpha(n-1)b^3 \\
&=&\frac{t-1}{2}(n^2-5n+6+2\alpha n -2\alpha)b^3\\
&=& \frac{t-1}{2}\left((n-{5 \over 2}+\alpha)^2-({1 \over 4}-3\alpha+\alpha^2)\right)b^3 \\
&=& \frac{t-1}{2}\left((1-{3 \over 2}b)^2b-({1 \over 4}-3\alpha+\alpha^2)b^3\right) \\
&\le& \frac{t-1}{2}\left((1-{3 \over 2}b)^2b+{7 \over 4}b^3\right) \\
&=& \frac{t-1}{2}b\left(1-3b+4b^2\right).
\end{eqnarray*}
\end{proof}

\begin{lemma}\label{M^3_t-local}
Let $t\geq 3$ be an integer and $b$ a real with $0<b<\frac{1}{3t-1}$.
Let $G$ be an $M^3_t$-free $3$-graph with $n\ge 3t$ vertices.
Then
$$\lambda_b(G)\leq
{t-1 \over 2}b(1-3b+6b^2).$$
\end{lemma}
\begin{proof}
Suppose $V(G)=[n]$. If no $b$-bounded feasible weight vector exists, then $\lambda_b(G)=0$ by definition and the
claim holds trivially. So assume that there exist $b$-bounded feasible weight vectors.
By Lemma \ref{b-bound-facts}, we may assume that $G$ has an optimum $b$-bounded weight vector $\vec{x}$
such that $G$ is $\vec{x}$-compressed, all vertices of $G$ have positive weights under $\vec{x}$, and such that
all pairs of vertices of weight less than $b$ are covered in $G$.
By relabeling the vertices of $G$ if needed we may assume that $x_1\geq\ldots \geq x_{n}>0$ and that $G$ is left-compressed relative to the natural order on $[n]$.

We use induction on $t$. For the basis step, let $t=3$. If $x_{n-1}=b$, then by Lemma \ref{almost-all-b},
$$\lambda(G,\vec{x})\leq b(1-3b+4b^2)\leq b(1-3b+6b^2).$$
Hence, we may assume that $x_{n-1}, x_{n}<b$.
By our assumption, $\{n-1,n\}$ is covered in $G$. Since $G$ is left-compressed, we have $\forall 2\leq i<j\leq n, 1ij\in G$.
Let $G'=G[\{2,\dots, n\}]$. By Lemma \ref{sub-structures}, $G'$ is $M^3_{2}$-free.   Since $x_2+\cdots+x_n=1-x_1=1-b$,
$\vec{y}=\frac{1}{1-b}(x_2,\ldots, x_n)$ is a $(\frac{b}{1-b})$-bounded feasible weight vector on $G'$. Let $b'=\frac{b}{1-b}$.
Since $b\leq \frac{1}{8}$, $b'=\frac{b}{1-b}\le \frac{1}{7}$.
Since $G'$ is $M^3_2$-free, and $\vec{y}$ is a $b'$-bounded feasible weight vector on $G'$,
by Lemma \ref{M^3_2-local-bound},
$$\lambda(G',\vec{x})=(1-b)^3\lambda(G', \vec{y})\leq (1-b)^3\cdot \frac{1}{2}b'(1-b')^2=\frac{1}{2}(1-b)^3 \frac{b}{1-b}\left(\frac{1-2b}{1-b}\right)^2
=\frac{1}{2}b(1-2b)^2.$$
Since the total contribution to $\lambda(G,\vec{x})$ from the edges containing $1$ is at most $\frac{1}{2}b(1-b)^2$, we have
$$\lambda(G,\vec{x})\leq \frac{1}{2}b(1-b)^2+\frac{1}{2}b(1-2b)^2=\frac{1}{2}b(2-6b+5b^2)<b(1-3b+6b^2).$$
Hence the claim holds. For the induction step, let $t\geq 4$. As before,
if $x_{n-1}=b$, then by Lemma \ref{almost-all-b},
$$\lambda(G,\vec{x})\leq \frac{t-1}{2}b(1-3b+4b^2)\leq \frac{t-1}{2}b(1-3b+6b^2).$$
Hence, we may assume that $x_{n-1}, x_{n}<b$. By our assumption, $\{n-1,n\}$ is covered in $G$.
Since $G$ is left-compressed we  have $\forall 2\leq i<j\leq n, 1ij\in G$. By Lemma \ref{sub-structures},
$G'=G[\{2,\dots, n\}]$ is $M^3_{t-1}$-free. Since
$\vec{y}=\frac{1}{1-b}(x_2,\ldots, x_n)$ is a $(\frac{b}{1-b})$-bounded feasible weight vector on $G'$,
by induction hypothesis,
$$\lambda(G',\vec{x})=(1-b)^3\lambda(G', \vec{y})\leq (1-b)^3\frac{t-2}{2}\frac{b}{1-b}\left(1-3\frac{b}{1-b}+6(\frac{b}{1-b})^2\right)=\frac{t-2}{2}b(1-5b+10b^2).$$
Since the total contribution to $\lambda(G,\vec{x})$ from the edges containing $1$ is at most $\frac{1}{2}b(1-b)^2$, we have
\begin{eqnarray*}
\lambda(G,\vec{x})&\leq& \frac{1}{2} b(1-b)^2+\frac{t-2}{2}b (1-5b+10b^2).\\
&=&\frac{1}{2}b[(t-1)-(5t-8)b +(10t-19)b^2]\\
&<&\frac{t-1}{2}  b (1-3b+6b^2),\\
\end{eqnarray*}
where the last inequality can be verified using the condition that $0<b\leq \frac{1}{3t-1}$ and $t\geq 4$.
\end{proof}

\begin{theo}\label{L^3_t-local}
Let $t\ge 2$ be an integer.  There exists a positive real $c_3=c_3(t)$ such that the following holds.
If $G$ is an $L^4_t$-free $4$-graph then $\lambda(G)\leq \lambda (K^4_{3t})=
\frac{(3t-1)(3t-2)(3t-3)}{24(3t)^3}.$ Furthermore, if $G$ also covers pairs and $G\neq K^4_{3t}$, then
$\lambda(G)\leq \lambda (K^4_{3t})-c_3$.
\end{theo}
\begin{proof} Since we may consider a dense subgraph covering pairs, it suffices to prove the
second statement.
Suppose that $G$ is on $[n]$. If $n\le 3t$, then the result holds obviously since $G\neq K^4_{3t}$.
Now suppose that $n\ge 3t+1$.
Let $\vec{x}$ be an optimum weight vector on $G$.
Without loss of generality, suppose
that  $x_1=\max \{x_i:i\in [n]\}$. Let $a=x_1$.
By Fact \ref{fact2}, we have $\lambda(G)= \frac{1}{4}\frac{\partial \lambda}{\partial x_1}$.
So it suffices to prove that
$\frac{\partial \lambda}{\partial x_1}\leq \frac{(3t-1)(3t-2)(3t-3)}{6\cdot (3t)^3}-c_3$ for some positive
real $c_3=c_3(t)$. Since $G$ is $L^4_t$-free, $L(1)$ is an $M^3_t$-free $3$-graph.
Since $G$ covers pairs, $L(1)$ is a $3$-graph on $[n]\setminus \{1\}$ that contains no
isolated vertex. Since $G$ covers pairs and $n\ge 3t+1$, by Lemma \ref{K-free}, $K^3_{3t-1}\not\subseteq L(1)$.
Let $\vec{y}=\frac{1}{1-a}(x_2,\dots, x_n)$. Then $\vec{y}$ is an $(\frac{a}{1-a})$-bounded feasible weight vector on $L(1)$.
We consider two cases.

\medskip

{\bf Case 1.} $a\geq \frac{1}{3t}$.

\medskip

Since $L(1)$ is $M^3_{t}$-free, $L(1)\neq K_{3t-1}^3$ and has no isolated vertex, by Theorem \ref{Mt-free},
$$\lambda(L(1),\vec{y})\leq \lambda(K^3_{3t-1})-c_1=\frac{(3t-1)(3t-2)(3t-3)}{6(3t-1)^3}-c_1.$$
Hence the claim holds by setting $c_3=c_1$.

\medskip

{\bf Case 2.} $a<\frac{1}{3t}$.

\medskip

Let $b=\frac{a}{1-a}$. Then $b<\frac{1}{3t-1}$.
By Lemma \ref{M^3_t-local}, we have
$$\frac{\partial \lambda}{\partial x_1}=(1-a)^3 \lambda(L(1), \vec{y})
\leq (1-a)^3 \frac{t-1}{2}b\left(1-3b+6b^2\right).$$
Substituting in $b=\frac{a}{1-a}$ and simplifying we get
\begin{eqnarray*}
\frac{\partial\lambda}{\partial x_1}
&\leq& (t-1)a(5a^2-{5 \over 2}a+{1 \over 2}).
\end{eqnarray*}

Let $f(a)=5a^3-\frac{5}{2} a^2 +\frac{1}{2} a$. Note that $f'(a)>0$ always.
So $f(a)$ is increasing. Since $a<\frac{1}{3t}$, we have
$$\frac{\partial \lambda}{\partial x_1}\leq (t-1)f(\frac{1}{3t})=
\frac{(t-1)(9t^2-15t+10)}{2(3t)^3}<\frac{(t-1)(3t-1)(3t-2)}{2\cdot (3t)^3}-c_3=\frac{(3t-1)(3t-2)(3t-3)}{6(3t)^3}-c_3,$$
for $t\geq 2$ and sufficiently small positive real $c_3=c_3(t)$.
\end{proof}

\medskip

\begin{coro}\label{CoroL^4_t}
$\pi_{\lambda}(L^4_t)=3! \lambda(K_{3t}^{4})={[3t]_3 \over (3t)^3}.$ \qed
\end{coro}
By Theorem \ref{L^3_t-local}, Corollary \ref{CoroL^4_t}, Corollary \ref{generalcoro} and Theorem \ref{BIJ-main}, we get the following result.
\begin{theo}\label{theoL^4_t}
Let $t\ge 2$ be an integer. Then $ex(n,H_{3t}^{L^4_t})= t_{3t}^{4}(n)$ for sufficiently large $n$. Moreover, if $n$ is sufficiently large and $G$ is an $H_{3t}^{L^4_t}$-free $4$-graph on $n$ vertices with $|G|=t_{3t}^{4}(n)$ edges, then $G=T_{3t}^{4}(n)$. \qed
\end{theo}


\section{Concluding remarks}

\medskip

Another natural way to extend the Hefetz-Keevash result in \cite{HK} is to establish the maximum Lagrangian
of an $r$-uniform intersecting family for $r\geq 4$, i.e. to determine the Lagrangian denisty of $M^r_2$,
for $r \geq 4$. The situation there is quite different from the $r=3$ case. Hefetz and Keevash \cite{HK}
conjectured that the maximum Lagrangian of an $r$-uniform intersecting family is achieved by a
feasible weight vector on the star $\{1ij: 2\leq i<j\leq n\}$. This conjecture was recently confirmed for all $r\geq 4$ by
Norin, Watts, and Yepremyan \cite{NWY}, who determined the Lagranigan density of $M^r_2$ as well as the stability of
the related Tur\'an problem. For the stability part of their result, see also \cite{Yep-thesis}.
Independently, Wu, Peng, and Chen \cite{WPC} had also confirmed the Hefetz-Keevash conjecture for $r=4$.

\end{document}